\documentclass[12pt]{amsart}
\usepackage[margin=1.35in]{geometry}
\usepackage{amscd,amsmath,amsxtra,amsthm,amssymb,stmaryrd,xr,mathrsfs,mathtools,enumerate,commath, comment}
\usepackage{stmaryrd}
\usepackage{xcolor}
\usepackage{commath}
\usepackage{comment}
\usepackage{tikz-cd}
\usepackage{longtable} 
\usepackage{pdflscape} 
\usepackage{booktabs}
\usepackage{hyperref}
\definecolor{vegasgold}{rgb}{0.77, 0.7, 0.35}
\definecolor{persianindigo}{rgb}{0.2, 0.07, 0.48}
\definecolor{crimsonglory}{rgb}{0.75, 0.0, 0.2}
\definecolor{airforceblue}{rgb}{0.36, 0.54, 0.66}
\definecolor{silver}{rgb}{0.75, 0.75, 0.75}
\definecolor{palesilver}{rgb}{0.79, 0.75, 0.73}
\definecolor{gold(metallic)}{rgb}{0.83, 0.69, 0.22}
\hypersetup{
 colorlinks=true,
 linkcolor=airforceblue,
 filecolor=brown,      
 urlcolor=gold(metallic),
 citecolor=crimsonglory,
 }
\newtheorem{lthm}{Theorem}

\usepackage[all,cmtip]{xy}

\DeclareFontFamily{U}{wncy}{}
\DeclareFontShape{U}{wncy}{m}{n}{<->wncyr10}{}
\DeclareSymbolFont{mcy}{U}{wncy}{m}{n}
\DeclareMathSymbol{\Sh}{\mathord}{mcy}{"58}
\usepackage[T2A,T1]{fontenc}
\usepackage[OT2,T1]{fontenc}

\newtheorem{theorem}{Theorem}[section]
\newtheorem{lemma}[theorem]{Lemma}

\newtheorem*{theorem*}{Theorem}
\newtheorem*{ass*}{Assumption}
\newtheorem{definition}[theorem]{Definition}
\newtheorem{corollary}[theorem]{Corollary}

\newtheorem{proposition}[theorem]{Proposition}

\newcommand{\cF}{\mathcal{F}}
\newcommand{\bK}{\mathbb{K}}

\newcommand{\Z}{\mathbb{Z}}
\newcommand{\p}{\mathfrak{p}}
\newcommand{\Q}{\mathbb{Q}}
\newcommand{\F}{\mathbb{F}}

\newcommand{\cO}{\mathcal{O}}

\newcommand{\op}[1]{\operatorname{#1}}

\newcommand\mtx[4] { \left( {\begin{array}{cc}
 #1 & #2 \\
 #3 & #4 \\
 \end{array} } \right)}

\numberwithin{equation}{section}

\begin{document}

\title[Statistics for $3$-isogeny Selmer groups]{Statistics for $3$-isogeny induced Selmer groups of elliptic curves}

\author[P.~Shingavekar]{Pratiksha Shingavekar}
\address[Shingavekar]{Chennai Mathematical Institute, H1, SIPCOT IT Park, Kelambakkam, Siruseri, Tamil Nadu 603103, India}
\email{pshingavekar@gmail.com}


\maketitle

\begin{abstract}
Given a sixth power free integer $a$, let $E_a$ be the elliptic curve defined by $y^2=x^3+a$. We prove explicit results for the lower density of sixth power free integers $a$ for which the $3$-isogeny induced Selmer group of $E_a$ over $\mathbb{Q}(\mu_3)$ has dimension $\leq 1$. The results are proven by refining the strategy of Davenport--Heilbronn, by relating the statistics for integral binary cubic forms to the statistics for $3$-isogeny induced Selmer groups.
\end{abstract}

\section{Introduction}
\subsection{Motivation and historical background}
\par The Mordell--Weil rank of an elliptic curve $E$ defined over a number field $F$ is of natural arithmetic interest. Given $n >0$, the $n$-Selmer group of $E/F$ gives an upper bound on the rank of the Mordell--Weil group $E(F)$. These Selmer groups can be viewed as \emph{random objects}, as the elliptic curves in question vary in natural families, ordered according to their discriminant. This average behaviour of Selmer groups have generated significant interest in recent years, thanks to the celebrated works of Bhargava--Shankar \cite{BS1,BS2,BS3}, Bhargava--Elkies--Schnidman \cite{bes}, Alp\"oge--Bhargava--Schnidman \cite{alpoge2022integers}, etc.
\par Let $E$ be an elliptic curve defined over $\Q$ with a rational $3$-isogeny $\varphi: E\rightarrow E'$. Then $E$ can be written either in the form $E_a:y^2=x^3+a$ where $a \in \Z$ or $E_{a,b}:y^2=x^3+a(x-b)^2$ where $a,b \in \Z$. Over the field $\bK:=\Q(\mu_3)$, the curve $E_a$ is $3$-isogenous to itself. The isogeny $\varphi:E_a\rightarrow E_a$ is explicitly given as follows \[\varphi(x,y)=\left( \frac{x^3+4a}{\varpi^2x^2}, \frac{y(x^3-8a)}{\varpi^3x^3} \right),\] where $\varpi:=1-e^{2\pi i/3}$. Associated with $\varphi$, the $\varphi$-Selmer group $\op{Sel}^\varphi(E_a/\bK)$ is a global Galois cohomology group satisfying local conditions, and sits in a short exact sequence  
\[ 0 \longrightarrow {E'(\bK)}/{\varphi(E(\bK))} \longrightarrow {\rm Sel}^\varphi(E/\bK) \longrightarrow \Sh(E/\bK)[\varphi] \longrightarrow 0.
\]
The advantage of studying these Selmer groups over the field $\bK$ is that even when the elliptic curve $E_a$ has rank $0$, one can still obtain interesting lower and upper bounds for the $\mathbb{F}_3$-dimensions of these Selmer groups in terms of $3$-ranks of class groups of certain quadratic extensions of $\bK$, cf. \cite{jhamajumdarshingavekar} for further details. In this article, we shall study distribution questions for these $\varphi$-Selmer groups over $\bK$, as $a$ ranges over all $6$-th power free integers.
\subsection{Main results}
We state our main results. Given a sixth power free integer $a$, set $L_a:=\Q(\mu_3, \sqrt{a})$ and let $r_3(L_a)$ be the $3$-rank of the class group of $L_a$. Note that the discriminant of $E_a$ is $\Delta_{E_a}:=-432a^2$. Given a real number $X>0$, we set $\mathcal{C}_1(X)$ to be the set of all isomorphism classes of elliptic curves $E_{a}$ with $|\Delta_{E_a}|\leq X$. We identify $\mathcal{C}_1(X)$ with 
\begin{equation}\label{def of C_1}\mathcal{C}_1(X):=\{a\in \Z\mid a \text{ is sixth power free and }432 a^2\leq X\}.\end{equation}
\begin{lthm}\label{thm A}[Theorem \ref{main density result}]
    There is a positive density of $6$-th power free integers $a$, such that 
    \[\dim_{\F_3}\op{Sel}^\varphi(E_a/\bK)\in \{r_3(L_a), r_3(L_a)+1\}.\] More precisely, 
    \[\liminf_{X\rightarrow \infty} \left(\frac{\#\left\{a\in \mathcal{C}_1(X)\mid \dim_{\F_3}\op{Sel}^\varphi(E_a/\bK)\in \{r_3(L_a), r_3(L_a)+1\}\right\}}{\# \mathcal{C}_1(X)}\right)\geq \frac{\zeta(6)}{64 \zeta(2)}.\]
\end{lthm}

\begin{lthm}\label{thm B}[Theorem \ref{main Selmer density thm}]
    There is a positive density of sixth power free integers $a$ such that $\dim_{\F_3}\op{Sel}^\varphi(E_a/\bK)\leq 1$. In greater detail, 
    \[\liminf_{X\rightarrow \infty} \frac{\# \left\{a\in \mathcal{C}_1(X)\mid \dim_{\F_3}\op{Sel}^\varphi(E_a/\bK)\leq 1 \right\}}{ \# \mathcal{C}_1(X)}\geq \frac{3 \zeta(6)}{128 \pi^2}.\]
\end{lthm}

\begin{lthm}\label{thm C}[Theorem \ref{main 3 Selmer thm}]
    There is a positive density of sixth power free integers $a$ such that $\dim_{\F_3}\op{Sel}_3(E_a/\bK)\leq 2$. More precisely, one has 
    \[\liminf_{X\rightarrow \infty} \frac{\# \left\{a\in \mathcal{C}_1(X)\mid \dim_{\F_3}\op{Sel}_3(E_a/\bK)\leq 2 \right\}}{ \# \mathcal{C}_1(X)}\geq \frac{3 \zeta(6)}{128 \pi^2}.\]
\end{lthm}

\subsection{Methodology and related work}
\par Certain upper and lower bounds for $\dim_{\mathbb{F}_3} \op{Sel}^\varphi(E_a/\bK)$ and $\dim_{\mathbb{F}_3} \op{Sel}_3(E_a/\bK)$ were recently obtained by Jha, Majumdar and Shingavekar \cite{jhamajumdarshingavekar}. These bounds are expressed in terms of certain $3$-ranks of class groups, as well as local terms. A combination of sieve methods along with the Davenport--Heilbronn \cite{davenportheilbronn1, davenportheilbronn2} statistics for integral binary cubic forms are used to prove our results. We note here that Bhargava, Elkies and Schnidman \cite{bes} have also proven related results for the average sizes of Selmer groups associated to $3$-isogenies. We point out that our results and approach are independent. 

\subsection{Organization}
Including the Introduction, the article consists of $4$ sections. \S\ref{s 2} is dedicated to recalling preliminary notions and setting up notation used throughout the article. We recall the relevant background regarding elliptic curves having a rational $3$-isogeny and their associated Selmer groups. We conclude the section by recalling some of the results of \cite{jhamajumdarshingavekar} that give bounds for these Selmer groups. In \S \ref{s 3}, we give an expository account of the results of Davenport-Heilbronn \cite{davenportheilbronn1, davenportheilbronn2}. We refine these results so that they can be applied in our setting. We prove a general result about sieves, cf. Proposition \ref{prop Phi asymptotic}. The main result of this section is Theorem \ref{DH modified thm}, which is a modification of \cite[Theorem 3]{davenportheilbronn2}. The main Theorems \ref{thm A}, \ref{thm B} and \ref{thm C} are proven in \S \ref{s 4}.

{\textbf{Acknowledgements:} I thank Anwesh Ray for his continuous support and invaluable suggestions throughout this work. This article would not have come to realisation without him.}

\section{Preliminaries} \label{s 2}
\par This section is dedicated to recalling relevant background as well as setting up notation. 
\subsection{Families of elliptic curves with rational $3$-isogeny}
\par Throughout this article, we let $\bK$ be the cyclotomic field $\Q(\mu_3)$. Set $\zeta_3:=\frac{-1+\sqrt{-3}}{2}$ and $\varpi:=1-\zeta_3$. Let $E_{/\Q}$ be an elliptic curve and $\Delta_E$ be the discriminant of $E$. Assume that there exists a rational $3$-isogeny $\varphi: E\rightarrow E'$ and let $\varphi': E'\rightarrow E$ be its dual. There is a Galois stable subgroup $C\subset E(\bar{\Q})$ of order $3$ which arises from a group-scheme $\mathcal{C}$ over $\Q$. We identify $E'$ with the quotient $E/\mathcal{C}$. Here, $C$ is a subgroup of the form $\{O,(\alpha, \beta),(\alpha, -\beta)\}$, where $\alpha \in \Q$ and $\beta^2 \in \Q$, cf. \cite[p.~3]{top}. Via a change of coordinates, we may assume that $\alpha =0$. Writing
\[E:y^2=x^3+rx^2+sx+t\text{ and }E':y^2=x^3+r'x^2+s'x+t'\] we find that $C=\{O,(0, \sqrt{t}),(0, -\sqrt{t})\}$, where $s^2=4rt$ and $t\neq 0$. There are thus two cases to consider. When $s=0$, the equation $s^2=4rt$ gives us that $r=0$ and thus $E$ is given by \[E_a:y^2=x^3+a,\] where $a:=t$. After replacing this curve with an isomorphic one, we assume that $a$ is a sixth power free integer. Such curves $E$ shall be referred to as Type I curves. \par In the second case when $s\neq 0$, the elliptic curve is uniquely given by 
\[E_{a,b}: y^2=x^3+a(x-b)^2,\] where $a, b$ are integers such that
\begin{itemize}
    \item $a\neq 0$, $b\neq 0$ and $d:=(4a+27b)\neq 0$, 
    \item $\op{gcd}(a,b)$ is squarefree. 
\end{itemize} 
Refer to such curves as \emph{Type II} curves. The discriminant of the curve $E_{a,b}$ is $\Delta_{a,b}:=-16 a^2 b^3 d$. Take $\mathcal{C}_2(X)$ to be the set of isomorphism classes of curves of Type II with absolute discriminant $\leq X$. In other words, 
\[\mathcal{C}_2(X):=\{(a, b)\in \Z^2\mid abd\neq 0\text{, } \op{gcd}(a,b)\text{ is squarefree and } 16 a^2 b^3 d\leq X\}.\]
\par For Type I and Type II elliptic curves, we describe the isogenies explicitly following \cite{velu}. We begin with curves of Type I. Note that $E_a$ has complex multiplication by $\cO_{\bK}$ and the $j$-invariant of $E_a$ is $0$. Assume that $27\nmid a$ and so in this setting, $E_a'=E_{-27 a}$. Further, note that $-27$ is a sixth power in $\bK$ and thus there is an isomorphism $\theta_\varpi:E_a \to E_{a}'$ defined over $\bK$ given by $\theta_\varpi(x,y)=(\varpi^{-2}x,\varpi^{-3}y)$. Thus, we have a $3$-isogeny $\varphi:E_a \to E_a$ defined over $\bK$ by 
\begin{equation}\label{def of phi}
    \varphi(x,y)=\left( \frac{x^3+4a}{\varpi^2x^2}, \frac{y(x^3-8a)}{\varpi^3x^3} \right).
\end{equation}

For the Type II curves $E_{a, b}$, via a change of variables $(x,y) \mapsto (\frac{x}{9}-\frac{4a}{3},\frac{y}{27})$ we identify $E_{a,b}/\mathcal{C}$ with ${E}'_{a,b}:=E_{-27a,d}$, where $d:=4a+27b$. Thus, we obtain a rational $3$-isogeny $\varphi:E_{a,b} \to {E}'_{a,b}$ given by
\begin{equation}\label{def of phiab}
\varphi(x,y) = \left( \frac{9(x^3+\frac{4}{3}ax^2-4abx+4ab^2)}{x^2}, \frac{27y(x^3+4abx-8ab^2)}{x^3} \right).
\end{equation}
and the corresponding dual isogeny ${\varphi}': {E}'_{a,b} \to E_{a,b}$ defined as follows
\begin{equation}\label{def of phi'ab}
{\varphi}'(x,y)  = \left( \frac{x^3-36ax^2 +108adx -108ad^2}{3^4x^2}, \frac{y(x^3-108adx+216ad^2)}{3^6x^3} \right).
\end{equation}
 For $c \in \bK^*$, note that the curves $E_{ac^2,bc^2}$ and $E_{a,b}$ are isomorphic over $\bK$ via the map $(x,y) \xmapsto{\theta_c} (c^{-2}x,c^{-3}y)$. Thus, without loss of generality, we can assume that $\op{gcd}(a,b)$ is square-free in $\bK$.
 \par In this article, we shall only be concerned with the family of curves of Type I. Similar questions for Type II curves shall be the subject of future work.
 
\subsection{Selmer groups and Tate--Shafarevich groups}
\par 
 The algebraic structure of the Mordell-Weil group of an elliptic curve can be studied using Galois theoretic techniques. Associated to an elliptic curve $E$, one considers Selmer groups, which are Galois cohomology groups with prescribed local conditions. In this subsection, we introduce the Selmer groups associated to an isogeny between elliptic curves. Let $F$ be a number field and denote by $\Sigma_F$ the set of places of $F$. Given $v \in \Sigma_F$, set $F_v$ to denote the completion of $F$ at $v$. Let $\bar{F}$ be a choice of algebraic closure of $F$ and likewise $\bar{F}_v$ an algebraic closure of $F_v$. Set $\op{G}_{F}$ (resp. $\op{G}_{F_v}$) to denote the absolute Galois group $\op{Gal}(\bar{F}/F)$ (resp. $\op{Gal}(\bar{F}_v/F_v)$). For each $v \in \Sigma_F$, choose an embedding $\iota_v: \bar{F} \hookrightarrow \bar{F}_v$. Note that this induces an embedding $\iota_v^*: \op{G}_{F_v} \hookrightarrow \op{G}_F$ of Galois groups.  Thus we get a restriction map on Galois cohomology
 \[\op{res}_v: H^1(\op{G}_F, \cdot)\rightarrow H^1(\op{G}_{F_v}, \cdot).\]

 Let $E$, $E'$ be elliptic curves over $F$ and $\varphi: E \to E'$ be an isogeny over $F$. For $\cF \in \{F, F_v\}$, consider the long exact sequence $\op{G}_{\cF}$-modules associated to 
 \[0\rightarrow E[\varphi]\rightarrow E\xrightarrow{\varphi} E'\rightarrow 0.\] This induces the following Kummer map \[\delta_{\cF}:E'(\cF) \longrightarrow E'(\cF)/\varphi(E(\cF)) \lhook\joinrel\xrightarrow{\overline{\delta}_{\cF}} H^1(\op{G}_{\cF}, E[\varphi]).\] From this, we arrive at the commutative diagram
 \begin{small}
\begin{center}\begin{tikzcd}
 0 \arrow[r] & E'(F)/\varphi(E(F)) \arrow[r, "\overline{\delta}_{ F}"] \arrow[d]
& H^1(\op{G}_F, E[\varphi]) \arrow[d, "\underset{v \in \Sigma_F}{\prod} {\rm res}_v"] \arrow[r] & H^1(\op{G}_F,E)[\varphi] \arrow[d] \arrow[r] & 0\\
 0 \arrow[r] & \underset{v \in \Sigma_F}{\prod} E'(F_v)/\varphi(E(F_v)) \arrow[r, "\underset{v \in \Sigma_F}{\prod} \overline{\delta}_{F_v}"] & \underset{v \in \Sigma_F}{\prod} H^1(\op{G}_{F_v},E[\varphi]) \arrow[r] & \underset{v \in \Sigma_F}{\prod} H^1(\op{G}_{F_v},E)[\varphi] \arrow[r] & 0.
\end{tikzcd}\end{center}
\end{small}
We now give the definition of the $\varphi$-Selmer group of $E$ over $F$.
\begin{definition}\label{mainsel}
With respect to notation above, we set
$${\rm Sel}^\varphi(E/F):= \{c \in H^1(G_F,E[\varphi]) \mid {\rm res}_v(c) \in {\rm Im } \ \delta_{F_v} \text{ for every } v \in \Sigma_F \}.$$
\end{definition}
On the other hand, recall that the \emph{Tate-Shafarevich group} of $E_{/F}$ is defined by \[\Sh(E/F):=\text{Ker}\big( H^1(G_F,E) \to \underset{v \in \Sigma_F}{\prod} H^1(G_{F_v},E) \big).\] 
The Selmer group is closely related to the Mordell--Weil group and the Tate--Shafarevich group, and one has the following short exact sequence
\begin{equation}\label{eq:defofsha}
0 \longrightarrow {E'(F)}/{\varphi(E(F))} \longrightarrow {\rm Sel}^\varphi(E/F) \longrightarrow \Sh(E/F)[\varphi] \longrightarrow 0
\end{equation}
When $\varphi=[n]: E(\overline{F}) \to E(\overline{F})$ is the multiplication by $n$ isogeny, the Selmer group ${\rm Sel}^\varphi(E/F)$ equals the $n$-Selmer group of $E$, which we denote by $\op{Sel}_n(E/F)$. In this setting, the above sequence specializes to 
\[0\rightarrow {E(F)}/{n E(F)} \rightarrow \op{Sel}_n(E/F) \rightarrow \Sh(E/F)[n]\rightarrow 0.\]

\subsection{Estimates for the Selmer group}
In this subsection, we recall the results of \cite{jhamajumdarshingavekar}, which give upper and lower estimates for the dimensions of $\op{Sel}^\varphi(E/\bK)$ for Type I elliptic curves $E_a$. In order to recall these results, we introduce further notation. Set $\widetilde{\Q}_{\widehat{\phi}_a}:= \frac{\Q[X]}{(X^2+27a)}$ and $\widetilde{\Q}_{{\phi}_a}:= \frac{\Q[X]}{(X^2-a)}$. Further, if $a \notin \bK^{*2}$, then set $L:=\frac{\bK[X]}{(X^2-a)}$.

Set
$$S_a:=\{ \mathfrak{q} \in \Sigma_{\bK} \mid a \in \bK_\mathfrak{q}^{*2} \text{ and } \upsilon_\mathfrak{q} (4a) \not\equiv 0 \pmod 6 \},$$
When $a \notin \bK^{*2}$, the set $S_a(L)$ consists of primes of $L$ lying above the primes in $S_a$.
Define $$S_a(\Q):=\{ \ \ell \in \Z \setminus \{3\} \ \mid \ -3a \in \Q_\ell^{*2} \ \text{ and } \ \upsilon_\ell(4a) \not\equiv 0 \pmod 6 \} \ \cup \ T_3(\Q),$$
where \[T_3(\Q):=\begin{cases} \{3\}, & \text{ if } -3a \in \Q_3^{*2} \text{ and } \upsilon_3(a)=1 \text{ or } 5, \\ \emptyset, & \text{ otherwise.} \end{cases}\]
When $a \neq n^2$ and $a \neq -3n^2$ for any $n \in \Z$, the sets $S_a(\widetilde{\Q}_{\widehat{\varphi}_a})$ and $S_a(\widetilde{\Q}_{{\varphi}_a})$ are the primes of $\widetilde{\Q}_{\widehat{\varphi}_a}$ and $\widetilde{\Q}_{{\varphi}_a}$ lying above the primes of $S_a(\Q)$, respectively.

\par Next, for a number field $\cF$ and a finite set $S$ of its finite primes, we denote by $r_3^S(\cF):=\op{dim}_{\mathbb{F}_3} \op{Cl}_S(\cF)[3]$, the dimension of the $3$-part of the $S$-class group of $\cF$. When $S$ is empty, we simply denote this quantity by $r_3(\cF)$. 
\begin{theorem}\label{type1bounds}
With respect to notation above, the following assertions hold.

\begin{enumerate}
    \item Suppose that  $a \not\in \bK^{*2}$  and  $a$  is sixth-power free in  $\bK$.  Then we have  \[\begin{split}r_3^{S_a(L)}(L) & \le \dim_{\F_3} {\rm Sel}^\varphi(E_a/\bK) \\
& \le \op{min}\{r_3^{S_a(L)}(L) + |S_a(L)| + 2,\quad r_3^{S_a(\widetilde{\Q}_{\widehat{\phi}_a})}(L)+r_3^{S_{a\alpha^2}(\widetilde{\Q}_{{\phi}_a})}(L)+|S_a(L)|+1\}.\end{split}\]
In particular, if $S_a=\emptyset$,  then $\dim_{\F_3} {\rm Sel}^\varphi(E_a/\bK) \in \{r_3(L), r_3(L)+1\}$ and it is uniquely determined by the root number of  $E_a/\Q$.
\item Suppose that $a \in \bK^{*2}$, we have  $\dim_{\F_3} {\rm Sel}^\phi(E_a/\bK) \le  |S_a|+1$. 
\end{enumerate}

\end{theorem}
\begin{proof}
    The above result is \cite[Theorem 3.23]{jhamajumdarshingavekar}.
\end{proof}
\section{Class group statistics}\label{s 3}
\par In this section, we discuss the method of Davenport \cite{davenport1, davenport2} and Davenport--Heilbronn \cite{davenportheilbronn1, davenportheilbronn2} who study the statistics for the $3$-parts of class numbers of quadratic number fields. Given real numbers $\xi <\eta$, let $N_3(\xi,\eta)$ be the number of conjugacy classes of cubic fields $K$ with discriminant $\Delta_K\in (\xi,\eta)$.

\begin{theorem}[Davenport--Heilbronn]
    With respect to notation above, 
    \[\begin{split}
        & \lim_{X\rightarrow \infty} X^{-1}N_3(0, X) = \left(12 \zeta(3)\right)^{-1},\\
        & \lim_{X\rightarrow \infty} X^{-1}N_3(-X, 0) = \left(4 \zeta(3)\right)^{-1}.\\
    \end{split}\]
\end{theorem}

\par Let $K/\Q$ be a cubic extension and $\widetilde{K}$ be the Galois closure of $K$. Then, either $K=\widetilde{K}$ and $\op{Gal}(K/\Q)\simeq \Z/3 \Z$, or $\op{Gal}(\widetilde{K}/\Q)\simeq S_3$. In the latter case, there is a triplet of fields that are conjugate to $K$. This triple is counted once. Let $\Phi$ be the set of all equivalence classes of irreducible primitive binary cubic forms 
\[F(x,y)=ax^3+b x^2y +c x y^2+d y^3,\] where $a,b,c, d\in \Z$. Two forms $F(x,y)$ and $G(x, y)$ are equivalent if there exists an integral matrix $A=\mtx{e}{f}{g}{h}$ with determinant $\pm 1$ which transforms $G(x,y)$ to $F(x,y)$. In other words, 
\[F(x,y)=G(ex + fy,gx+h y).\] The discriminant of $F(x,y)$ is 
\[D = b^2c^2+18abcd -27 a^2 d^2-4b^3 d-4 c^3 a.\]
Let $K$ be a cubic field over $\Q$ and let $1, \omega, \nu$ be an integral basis of $K$. For $\alpha\in K$, let $\mathfrak{d}(\alpha)$ denote the discriminant of $\alpha$. Then following \cite{davenport1}, set 
\[F_K(x,y):= \left(\frac{\mathfrak{d}(\omega x+\nu y)}{\Delta_K}\right)^{\frac{1}{2}}. \]
The form $F_K$ is irreducible and integral. Moreover, it satisfies local conditions at each prime $p$, which we proceed to now describe. Given $F(x,y), G(x,y)\in \Phi$ and a natural number $m$, we write 
\[F(x,y)\equiv G(x,y)\pmod{m}\] if the coefficients of $F(x,y)$ are congruent to those of $G(x,y)$ modulo $m$. Given a prime $p$ define a symbol $(F, p)$ as follows. Write $(F,p)=(111)$ to mean that
\[F(x,y)\equiv \lambda_1(x,y)\lambda_2(x,y) \lambda_3(x,y)\pmod{p},\]
where $\lambda_1,\lambda_2$ and $\lambda_3$ are three linear forms that have no constant quotient. Likewise, write $(F,p)=(12)$ to mean that 
\[F(x,y)\equiv \lambda(x,y)\kappa(x,y)\pmod{p},\] where $\lambda(x,y)$ is a linear form and $\kappa(x,y)$ is an irreducible quadratic form. Write $(F,p)=(3)$ to mean that $F(x,y)$ itself is irreducible modulo $p$. There are two ramified cases, first $(F,p)=(1^2 1)$ means that 
\[F(x,y)=\lambda_1(x,y)^2 \lambda_2(x,y)\pmod{p},\] where $\lambda_i(x,y)$ are linear forms with no constant quotient. Finally, write $(F,p)=(1^3)$ to mean that 
\[F(x,y)=\alpha\lambda(x,y)^3\pmod{p},\] where $\alpha$ is a constant and $\lambda$ is a linear form. 
\begin{lemma}\label{ramification T_p lemma}
    A prime $p$ factorizes in $K$ according to the following table
   \[ \begin{split}
        & (p)=\p_1\p_2\p_3 
 \quad \text{ if }\quad (F_K,p)=(111),\\
        & (p)=\p_1\p_2 \quad\text{ if }\quad (F_K,p)=(12),\\
        & (p)=(p)\quad \text{ if }\quad (F_K,p)=(3),\\
        & (p)=\p_1^2 \p_2 \quad\text{ if }\quad (F_K,p)=(1^2 1),\\
        & (p)=\p^3\quad\text{ if }\quad (F_K,p)=(1^3).\\
    \end{split}\]
\end{lemma}
\begin{proof}
    This result is \cite[Lemma 11]{davenportheilbronn2}.
\end{proof}

For $\alpha\in \{(111), (12), (3), (1^2 1), (1^3)\}$ we let $T_p(\alpha)$ denote the set of $F \in \Phi$ for which $(F,p)=\alpha$. 
\begin{definition}\label{def of Up and Vp}
    For each prime $p$, define subsets $U_p$ and $V_p$ of $\Phi$ as follows. The definition for $p=2$ is different from that of $p>2$. The set $V_2$ consists of all $F$ such that $F\in \Phi$ such that $D\equiv 1\pmod{4}$ or $D\equiv 8, 12\pmod{16}$. For $p>2$, $V_p$ consists of all $F$ for which $p^2\nmid D$. The set $U_p$ consists of $F$ such that either of the following conditions are satisfied
    \begin{itemize}
        \item $F\in V_p$, 
        \item $(F,p)=(1^3)$ and the congruence $F(x,y)\equiv ep \pmod{p^2}$ has a solution for some $e\not\equiv 0\mod{p}$.
    \end{itemize}
    With respect to notation above, we set $U:=\bigcap_p U_p$ and $V:=\bigcap_p V_p$. 
\end{definition}
We note that for $p>2$, $U_p$ and $V_p$ are determined by congruence conditions modulo $p^2$. On the other hand, for $p=2$, the condition is determined by congruence conditions modulo $2^4$. Given an integer $m=\prod_p p^{r_p}$, let $\Phi(m)$ denote the set of all forms 
\[a x^3+ b x^2 y+c x y^2+ d y^3,\] with $a,b,c,d\in \Z/m\Z$ such that at least one of the coefficients $a,b,c$ or $d$ is not divisible by $p$ for each prime divisor of $m$. An elementary calculation shows that 
\[\# \Phi(m)= \prod_{p|m} \# \Phi(p^{r_p})=\prod_{p|m} \left(p^{4r_p}-p^{4r_p-4}\right)=\prod_{p|m} p^{4r_p}\left(1-p^{-4}\right).\] Let $S$ be a subset of $\Phi$. We say that $S$ is defined by conditions modulo $m$ if there is a set of residue classes $S(m)\subset \Phi(m)$ such that $S$ consists of all forms that reduce to a class in $S(m)$ modulo-$m$. 
\begin{definition}
    Suppose that $S$ is defined by congruence classes $S(m)$ modulo $m$. Given real numbers $\xi<\eta$, let $N(\xi, \eta, S)$ be the number of forms $F(x,y)\in S$ with $D\in (\xi, \eta)$. 
\end{definition}
Suppose that $S$ is defined by congruence classes $S(m)$ modulo $m$. Define \[\mathfrak{d}(S) :=\frac{\# S(m)}{\# \Phi(m)}=\left( \frac{\# S(m)}{\prod_{p|m}p^{4r_p}\left(1-p^{-4}\right)} \right). \]
\begin{theorem}[Davenport]
    Suppose that $S$ is defined by congruence classes $S(m)$ modulo $m$. Then, the following assertions hold
    \begin{enumerate}
        \item $\lim_{X\rightarrow \infty} \frac{N(0, X, S)}{ X }=\frac{5}{4} \pi^{-2} \mathfrak{d}(S)$,
        \item $\lim_{X\rightarrow \infty} \frac{N(-X, 0, S)}{ X }=\frac{15}{4} \pi^{-2} \mathfrak{d}(S)$.
    \end{enumerate}
\end{theorem}

\begin{proof}
    The result is direct consequence of \cite{davenport1, davenport2}, see \cite[p. 414, l. 11--13]{davenportheilbronn2}.
\end{proof}
\begin{lemma}
    For a prime $p$, we have
    \begin{equation*}
    \begin{split}
       & \mathfrak{d}(T_p(111))  =\frac{1}{6}p(p-1)(p^2+1)^{-1},\\
       & \mathfrak{d}(T_p(12))  =\frac{1}{2}p(p-1)(p^2+1)^{-1},\\
       & \mathfrak{d}(T_p(3))  =\frac{1}{3}p(p-1)(p^2+1)^{-1},\\
       & \mathfrak{d}(T_p(1^2 1))   =p(p^2+1)^{-1},\\
       & \mathfrak{d}(T_p(1^3)) =(p^2+1)^{-1}.
         \end{split}
   \end{equation*}
\end{lemma}
\begin{proof}
    The result is \cite[Lemma 1]{davenportheilbronn2}.
\end{proof}

\begin{lemma}
    Let $U_p$ and $V_p$ be as in Definition \ref{def of Up and Vp}. Then we have
    \begin{enumerate}
        \item $\mathfrak{d}(V_p)=(p^2-1)(p^2+1)^{-1}$ for all $p$,
        \item $\mathfrak{d}(U_p)=(p^3-1)p^{-1}(p^2+1)^{-1}$ for all $p$.
    \end{enumerate}
\end{lemma}
\begin{proof}
    This result is \cite[Lemma 4, 5]{davenportheilbronn2}.
\end{proof}

\begin{definition}
   We define a new set of conditions $V_p'$ for all primes $p$. The definition for $p\in \{2,3\}$ differs from that for $p\geq 5$. We say that $F$ satisfies $V_2'$ if $2 \nmid D$. On the other hand, $F$ satisfies $V_3'$ if it satisfies $V_3$ and $3\mid D$. 
    For $p \ge 5$, we define $V_p':=V_p$. With these definitions in place $V':= \bigcap_{p} V_p'$. Thus, $V'$ consists of all $F\in \Phi$ such that for all primes $p$, $F$ satisfies $V_p'$. 
\end{definition}

\begin{proposition}
    With respect to the notation above
    \begin{equation*}
        \begin{split}
            & \mathfrak{d}(V_2')=\frac{1}{5},\\
            & \mathfrak{d}(V_3')=\frac{2}{5}.
        \end{split}
    \end{equation*}
    
\end{proposition}
\begin{proof}
    Notice that $V_2'$ consists of all $F \in \Phi$ for which $2 \nmid D$. This means $V_2'= \bigcup_{\alpha} T_2(\alpha)$, where $\alpha \in A:=\{(111),(12),(3)\}$.
   Therefore,
   \begin{equation*}
       \begin{split}
           \mathfrak{d}(V_2') & = \sum_{\alpha \in A} \mathfrak{d}(T_2(\alpha))\\
            & =  \left( \frac{1}{6}\frac{2(2-1)}{(2^2+1)} + \frac{1}{2}\frac{2(2-1)}{(2^2+1)} + \frac{1}{3}\frac{2(2-1)}{(2^2+1)} \right)\\
            & = \frac{2}{5}.
       \end{split}
   \end{equation*}
   Whereas for the prime $p=3$, we have that $V_3 \setminus V_3'$ consists of all $F \in \Phi$ for which $3 \nmid D$. This means $V_3 \setminus V_3'= \bigcup_{\alpha} T_3(\alpha)$, where $\alpha \in A:=\{(111),(12),(3)\}$. Then we have 
    \begin{equation*}
        \begin{split}
            \mathfrak{d}(V_3') & = \mathfrak{d}(V_3)-\sum_{\alpha \in A} \mathfrak{d}(T_3(\alpha))\\
            & = \frac{(3^2-1)}{(3^2+1)} - \left( \frac{1}{6}\frac{3(3-1)}{(3^2+1)} + \frac{1}{2}\frac{3(3-1)}{(3^2+1)} + \frac{1}{3}\frac{3(3-1)}{(3^2+1)} \right)\\
            & = \frac{(3-1)}{(3^2+1)}\\
            & = \frac{1}{5}.
        \end{split}
    \end{equation*}
\end{proof}

Given a cubic field $K$, we define a map $\Lambda: K \to F_K$
from the conjugate classes of cubic fields by $K \mapsto F_K(x,y)$, where $F_K(x,y)$ is the integral binary cubic form.
\begin{proposition}\label{bijection to U}
    The map $\Lambda$ defined above is a discriminant preserving bijection from the set of conjugacy classes of cubic fields to $U$.
\end{proposition}
\begin{proof}
    The result is \cite[Proposition 4]{davenportheilbronn2}.
\end{proof}
Given a quadratic number field $k$, set $h_3^*(k)$ to be the order of the $3$-torsion subgroup of the class group of $k$. Let $\Delta_2$ be the discriminant of a quadratic field, set $h_3^*(\Delta_2):=h_3^*\left(\Q(\sqrt{\Delta_2})\right)$.
\begin{lemma}\label{h_3 star lemma}
    For real numbers $\xi < \eta$, we have \begin{equation}\label{h_3 star formula}\frac{1}{2} \sum_{\substack{\xi < \Delta_2 < \eta \\ 2 \nmid \Delta_2, 3 \mid \Delta_2}} (h_3^*(\Delta_2)-1) = N(\xi, \eta, V'),\end{equation}
    where $\Delta_2$ ranges over discriminants of quadratic number fields.
\end{lemma}

\begin{proof}
    The proof is similar to that of \cite[Theorem 3]{davenportheilbronn2}. We first analyze the left hand side \eqref{h_3 star formula}. Let $K$ be a cubic field in which no prime ramifies completely. Let $\widetilde{K}$ be the Galois closure of $K$ over $\Q$. Since no prime splits completely in $K$, it follows that $\widetilde{K}$ is an $S_3$-extension of $\Q$. Let $k$ be the quadratic extension of $\Q$ which is contained in $\widetilde{K}$. Then we have that $\Delta_K=\Delta_2$, where $\Delta_2$ is the discriminant of $k$, see \cite[p.419, l.15 and proof of Lemma 12]{davenportheilbronn2}. Let $p$ be prime of $\Q$ and $\mathcal{P}|p$ be a prime of $\widetilde{K}$. Then the ramification index $e(\mathcal{P}/p)$ is $1$, $2$ or $4$. In particular, it is not divisible by $3$. Therefore, any prime of $k$ is unramified in $\widetilde{K}$. Hence, $\widetilde{K}$ is contained in the Hilbert class field of $k=\Q(\sqrt{\Delta_2})$. Thus, $\op{Gal}(\widetilde{K}/k)$ can be viewed as a $\Z/3\Z$ quotient of the class group of $k$. There are $\frac{1}{2}\left(h_3^*(\Delta_2)-1\right)$ such $\Z/3\Z$-quotients. This implies that the number of conjugate triplets of cubic fields $K$ in which no prime ramifies completely is given by $\frac{1}{2}(h_3^*(\Delta_2)-1)$. We have thus shown that 
    \begin{equation}\label{N' eqn}\frac{1}{2} \sum_{\substack{\xi < \Delta_2 < \eta \\ 2 \nmid \Delta_2, 3\mid \Delta_2}} (h_3^*(\Delta_2)-1) = N_3'(\xi, \eta),\end{equation} where $N_3'(\xi, \eta)$ is the number of conjugate triples of cubic fields $K/\Q$ in which no prime is totally ramified and such that $2 \nmid \Delta_K$ and $3\mid \Delta_K$.
    \par In order to complete the proof, it suffices to show that $N_3'(\xi, \eta)=N(\xi, \eta, V')$. Recall from Proposition \ref{bijection to U} that there is a bijection
    \[\Lambda: \{K/\Q \mid [K:\Q]=3\}/_{\simeq}\xrightarrow{\sim} U,\] between conjugacy classes of cubic extensions and $U$, defined by $\Lambda(K):=F_K(x,y)$. Lemma \ref{ramification T_p lemma} implies that a prime $p$ is not totally ramified in $K$ if and only if $\Lambda(K)\notin T_p(1^3)$. Note that $V_p=U_p\backslash T_p(1^3)$. On the other hand, $2 \nmid \Delta_2$ (resp. $3\mid \Delta_2$) if and only if $2 \nmid \Delta_K$ (resp. $3\mid \Delta_K$), since $\Delta_K=\Delta_2$. Hence, $\Lambda$ restricts to a discriminant preserving bijection 
    \begin{equation}\label{restricted bijection}\Lambda: \{K/\Q \mid [K:\Q]=3, \text{ no prime is totally ramified in }K\text{, }2 \nmid \Delta_K, 3 \mid \Delta_K\}/_{\simeq}\xrightarrow{\sim} V'.\end{equation}
    From \eqref{N' eqn} and \eqref{restricted bijection}, we deduce that
    \[\frac{1}{2} \sum_{\substack{\xi < \Delta_2 < \eta \\2 \nmid \Delta_2, 3\mid \Delta_2}} (h_3^*(\Delta_2)-1)=N'_3(\xi, \eta)=N(\xi, \eta, V'),\] which completes the proof.
 \end{proof}

 \begin{proposition}\label{boring prop 1}
     With respect to notation above, we have that
     \begin{enumerate}
         \item $\lim_{X\rightarrow \infty}\left(\frac{N(0, X, V')}{X}\right)=\frac{1}{12\pi^2}$,
          \item $\lim_{X\rightarrow \infty}\left(\frac{N(-X, 0, V')}{X}\right)=\frac{1}{4 \pi^2}$.
     \end{enumerate}
 \end{proposition}

\begin{proof}
    It follows from arguments identical to the proof of \cite[Proposition 3]{davenportheilbronn2} that 
    \[\begin{split}
         \lim_{X\rightarrow \infty}\left(\frac{N(0, X, V')}{X}\right) &= \mathfrak{d}(V_2')\times \mathfrak{d}(V_3')\times \left(\prod_{p \neq 2,3} \mathfrak{d}(V_p, p^2)\right)\times  \frac{5}{4 \pi^2}\\
         & = \frac{\mathfrak{d}(V_2', 2^2)}{\mathfrak{d}(V_2, 2^2)} \times \frac{\mathfrak{d}(V_3', 3^2)}{\mathfrak{d}(V_3, 3^2)} \times \frac{1}{2 \pi ^2},\\
          \lim_{X\rightarrow \infty}\left(\frac{N(-X, 0, V')}{X}\right) &=\mathfrak{d}(V_2', 2^2) \times \mathfrak{d}(V_3', 3^2)\times \left(\prod_{p\neq 2,3} \mathfrak{d}(V_p, p^2)\right)\times  \frac{15}{4 \pi^2},\\
           & = \frac{\mathfrak{d}(V_2', 2^2)}{\mathfrak{d}(V_2, 2^2)}\times \frac{\mathfrak{d}(V_3', 3^2)}{\mathfrak{d}(V_3, 3^2)}\times \frac{3}{2 \pi ^2}.\\
    \end{split}\]

    Noting that
\begin{equation*}
    \begin{split}
    & \frac{\mathfrak{d}(V_2', 2^2)}{\mathfrak{d}(V_2, 2^2)}=\left(\frac{2}{5}\right)\times \left(\frac{2^2-1}{2^2+1}\right)^{-1}=\frac{2}{3},\\
    & \frac{\mathfrak{d}(V_3', 3^2)}{\mathfrak{d}(V_3, 3^2)}=\left(\frac{1}{5}\right)\times \left(\frac{3^2-1}{3^2+1}\right)^{-1}=\frac{1}{4},  
    \end{split}
\end{equation*}
    we obtain the result.
\end{proof}
It conveniences us to prove a general result about the density of a set of integers defined by congruence conditions at every prime $p$. When there are only finitely many primes at which local conditions are defined, the density is easy to work out, and is a consequence of the Chinese remainder theorem. When there are infinitely many conditions, we require that some additional conditions are satisfied, which we now describe in greater detail.
\par Let $c,d$ be real numbers such that $c<d$ and $\Omega(X)$ be the set of integers $a\in (c X, d X)$. At each prime $p$, we consider a local condition $\Phi_p$ defined by a finite set of congruence classes $\{\phi_1^{(p)}, \dots, \phi_t^{(p)}\}$ modulo $p^{r_p}$. Setting \[\mathfrak{d}(\Phi_p):=\frac{t}{p^{r_p}},\]say that $a\in \Phi_p$ if $a\mod{p^{r_p}}$ belongs to the set $\{\phi_1^{(p)}, \dots, \phi_t^{(p)}\}$. We define $\Phi_p(X)$ to be the set of $a\in \Omega(X)$ such that $a$ satisfies $\Phi_p$. Let $\Phi(X)$ be the set of $a\in \Omega(X)$ satisfying the conditions $\Phi_p$ for all primes $p$. In other words, 
\[\Phi(X):=\bigcap_p \Phi_p(X),\] where $p$ ranges over all prime numbers.

\begin{proposition}\label{prop Phi asymptotic}
    With respect to notation above, suppose that 
    \begin{enumerate}
        \item $\sum_{p} \left(1-\mathfrak{d}(\Phi_p)\right)<\infty$,
        \item the product $\prod_p \mathfrak{d}(\Phi_p)$ converges.
    \end{enumerate}
     Then, we have that 
    \[\# \Phi(X)\sim (d-c)\left(\prod_p \mathfrak{d}(\Phi_p)\right) X.\]
\end{proposition}

\begin{proof}
Given a positive real number $Z$, set $\Phi_{Z}(X)$ to be the set of integers $a\in \Omega(X)$ such that $a$ satisfies $\Phi_p$ for all $p\leq Z$. Since there are only finitely many local conditions defining $\Phi_Z(X)$, we have that 
 \[\lim_{Y\rightarrow \infty} \frac{\# \Phi_Z(X)}{(d-c)X}=\prod_{p\leq Z} \mathfrak{d}(\Phi_p).\]
 On the other hand, since $\Phi(X)$ is contained in $\Phi_Z(X)$, we get that 
 \[\limsup_{X\rightarrow \infty} \frac{\# \Phi(X)}{(d-c) X}\leq \lim_{X\rightarrow \infty} \frac{\# \Phi_Z(X)}{(d-c)X}=\prod_{p\leq Z} \mathfrak{d}(\Phi_p).\] Thus, taking $Z\rightarrow \infty$, we find that 
 \[\limsup_{X\rightarrow \infty} \frac{\# \Phi(X)}{(d-c)X} \leq \prod_{p} \mathfrak{d}(\Phi_p).\]
Given a prime $p$, set $\Phi_p'(X)$ to be the set of integers $a\in \Omega(X)$ such that $a$ does not satisfy $\Phi_p$. Since \[\Phi_Z(X)\subseteq \Phi(X)\cup \left(\bigcup_{p>Z } \Phi_p'(X)\right),\] we find that 
\[\lim_{X\rightarrow \infty} \frac{\# \Phi_Z(X)}{(d-c)X}\leq \liminf_{X\rightarrow \infty} \frac{\# \Phi(X)}{(d-c)X}+ \sum_{p>Z} \limsup_{Y\rightarrow \infty} \frac{\# \Phi_p'(X)}{(d-c)X}.\]
Note that $\# \Phi_p'(X)= \left(1-\mathfrak{d}(\Phi_p) \right) (d-c)X +O(1)$. Thus, we find that 
\[\sum_{p>Z} \limsup_{X\rightarrow \infty} \frac{\# \Phi_p'(X)}{(d-c)X}\leq \sum_{p\geq Z}  \left(1-\mathfrak{d}(\Phi_p)\right).\] This quantity goes to $0$ as $Z\rightarrow \infty$. On the other hand, 
\[\lim_{Z\rightarrow \infty}\left(\lim_{X\rightarrow \infty} \frac{\# \Phi_Z(X)}{(d-c)X}\right)=\prod_p \mathfrak{d}(\Phi_p).\]
 Thus, we have shown that 
 \[\liminf_{X\rightarrow \infty} \frac{\# \Phi(X)}{(d-c) X}=\prod_p \mathfrak{d}(\Phi_p).\]     

 Since $\liminf_{X\rightarrow \infty} \frac{\# \Phi(X)}{(d-c) X}$ and $\limsup_{X\rightarrow \infty} \frac{\# \Phi(X)}{(d-c) X}$ both exist and are equal to $\prod_p \mathfrak{d}(\Phi_p)$, we conclude that 
  \[\# \Phi(X)\sim (d-c)\left(\prod_p \mathfrak{d}(\Phi_p)\right) X.\]
\end{proof}

\begin{definition}
    Given a positive real number $X$, let $\mathcal{M}^+(X)$ (resp. $\mathcal{M}^-(X)$) be the set of all discriminants $\Delta_2 \in (0,X)$ (resp. $\Delta_2\in (-X, 0)$) such that $2 \nmid \Delta_2$ and $3 \mid \Delta_2$. 
\end{definition}

\begin{lemma}\label{M pm asymptotic 1/2 pi^2}
    For $\mathcal{M}^{\pm}(X)$ defined as above, we have that \[\# \mathcal{M}^{\pm}(X) \sim \frac{1}{2\pi^2} X.\]
\end{lemma}
\begin{proof}
    Note that $\Delta_2\in (-X, 0)$ satisfies $\mathcal{M}^-(X)$ if and only if
    \begin{itemize}
        \item $\Delta_2\equiv 3\pmod{4}$,
        \item $\Delta_2\equiv 3, 6\pmod{9}$,
        \item $p^2\nmid \Delta_2$ for all primes $p\geq 5$.  
        \end{itemize} 
        Let $\Phi_p$ be the congruence conditions defined as above. We find that 
        \[\mathfrak{d}(\Phi_p)=\begin{cases}
            \frac{1}{4} & \text{ if } p=2;\\
            \frac{2}{9} & \text{ if } p=3;\\
             1-\frac{1}{p^2} & \text{ if } p\geq 5. \\
        \end{cases}\]
Therefore, it follows from Proposition \ref{prop Phi asymptotic} that 
 \[ \begin{split}
   \mathcal{M}^-(X) &\sim \left(\frac{1}{4} \times \frac{2}{9} \times \prod_{p \ge 5} \left(1 - \frac{1}{p^2}\right) \right) X\\
    & = \left( \frac{1}{4} \times \frac{2}{9} \times \frac{1}{(1-\frac{1}{4})} \times \frac{1}{(1-\frac{1}{9})} \times \frac{1}{\zeta(2)} \right) X \\
    & = \left(\frac{1}{12} \times \frac{6}{\pi^2}\right) X\\
    & = \frac{X}{2\pi^2}.
    \end{split}\]
    The result for $\mathcal{M}^+(X)$ follows similarly.
\end{proof}

\begin{theorem}\label{DH modified thm}
    Let $\Delta_2$ denote a quadratic discriminant. For $\mathcal{M}^+(X)$ and $\mathcal{M}^-(X)$ defined as above, we have that
   \[ \begin{split}
        & \sum_{\Delta_2 \in \mathcal{M}^+(X)} h_3^*(\Delta_2)\sim \frac{4}{3} \sum_{\Delta_2 \in \mathcal{M}^+(X)} 1,\\
        & \sum_{\Delta_2 \in \mathcal{M}^-(X)} h_3^*(\Delta_2)\sim 2 \sum_{\Delta_2 \in \mathcal{M}^-(X)} 1.
    \end{split}\]   
\end{theorem}

\begin{proof}
We have that 
    \[\frac{1}{2} \sum_{\Delta_2 \in \mathcal{M}^+(X)} (h_3^*(\Delta_2)-1) = N(0, X, V'), \] from  Lemma \ref{h_3 star lemma} and $N(0, X,  V')\sim \frac{X}{12\pi^2}$ from Proposition \ref{boring prop 1}. Therefore, 
    \[\begin{split} & \lim_{X\rightarrow \infty}\left(\frac{\sum_{\Delta_2 \in \mathcal{M}^+(X)} h_3^*(\Delta_2)}{\sum_{\Delta_2 \in \mathcal{M}^+(X) } 1}\right)\\ = & \lim_{X\rightarrow \infty} \left(\frac{2 N(0, X, V')+ \sum_{\Delta_2 \in \mathcal{M}^+(X) } 1}{\sum_{\Delta_2 \in \mathcal{M}^+(X) } 1}\right) \\ 
    =& \lim_{X\rightarrow \infty} \left(\frac{\frac{X}{6\pi^2}+ \frac{X}{2\pi^2}}{\frac{X}{2\pi^2}}\right) =\frac{4}{3}.
 \end{split}\]

 The same argument gives 
 \[\lim_{X\rightarrow \infty}\left(\frac{\sum_{\Delta_2 \in \mathcal{M}^-(X)} h_3^*(\Delta_2)}{\sum_{\Delta_2 \in \mathcal{M}^-(X) } 1}\right)=2.\]
\end{proof}

\section{Main results}\label{s 4}

\par In this section, we prove the main results of this article. Let $a$ be a nonzero sixth power free integer and $E_a$ be the associated elliptic curve of Type I. For $Y>0$, set 
\[\mathcal{S}(Y):=\left\{ a \mid a\text{ is sixth power free and }a\leq Y\right\}.\]
Recall that 
\[\mathcal{C}_1(X):=\{a\in \Z\mid a \text{ is sixth power free and }432 a^2\leq X\}.\] It is easy to see that $\mathcal{C}_1(X)=\mathcal{S}(\frac{\sqrt{X}}{12\sqrt{3}})$. We consider the following sets 
\[\begin{split} & \mathcal{C}_{1}'(X):=\left\{ a\in \mathcal{C}_1(X)\mid \op{dim}_{\F_3} \op{Sel}^\varphi(E_a/\bK)\leq 1\right\}, \\ 
& \mathcal{S}'(Y):=\left\{ a\in \mathcal{S}(Y)\mid \op{dim}_{\F_3} \op{Sel}^\varphi(E_a/\bK)\leq 1\right\}. \\ \end{split}\]

The density of curves of Type I for which $\op{dim}_{\F_3} \op{Sel}^\varphi(E_a/\bK)\leq 1$ is defined as follows 
\[\begin{split}\mathfrak{d}_{\leq 1}=&  \lim_{X\rightarrow \infty} \left(\frac{\# \{a\in \mathcal{C}_1(X)\mid \op{dim}_{\F_3} \op{Sel}^\varphi(E_a/\bK)\leq 1\}}{ \# \mathcal{C}_1(X)}\right) \\
= &\lim_{X\rightarrow \infty} \left(\frac{\# \mathcal{C}_{1}'(X)}{ \# \mathcal{C}_1(X)}\right).\\ \end{split}\] Upon replacing $Y$ with $\frac{\sqrt{X}}{12\sqrt{3}}$, it is clear that
\[\mathfrak{d}_{\leq 1}=\lim_{Y\rightarrow \infty} \left(\frac{\# \mathcal{S}' (Y)}{\# \mathcal{S}(Y)}\right).\] 

Let $\mathcal{T}(Y)$ be the set of $a\in \mathcal{S}(Y)$ such that 
\begin{itemize}
    \item $a>0$, 
    \item $3\nmid a$,
    \item $2^4|a$ and $2^5\nmid a$,
    \item write $a=16 a'$, then, $a'$ is squarefree.
\end{itemize} 

\begin{proposition}\label{dim is 0 or 1}
    For $a\in \mathcal{T}(Y)$, we have that \[\dim_{\F_3} \op{Sel}^\varphi(E_a/\bK)\in \{r_3(L), r_3(L)+1\}.\]
\end{proposition}

\begin{proof}
    It suffices to show that $S_a=\emptyset$. Recall that $$S_a:=\{ \mathfrak{q} \in \Sigma_{\bK} \mid a \in {\bK}_\mathfrak{q}^{*2} \text{ and } \upsilon_\mathfrak{q} (4a) \not\equiv 0 \pmod 6 \}.$$ Let $q$ be the rational prime such that $\mathfrak{q}|q$. First consider the case when $q\neq 2,3$. In this case, $a$ is not divisible by $q^2$ since it is squarefree. Moreover, $q$ is unramified in $\bK$. Suppose that $v_\mathfrak{q}(4a)\not \equiv 0\mod{6}$. Then, in particular, this would imply that $q$ divides $a$ exactly once. Let $\pi_\mathfrak{q}$ be the uniformizer of ${\bK}_\mathfrak{q}$. We find that $\pi_\mathfrak{q}$ also divides $a$ exactly once and hence, $a\notin ({\bK}_\mathfrak{q}^*)^2$, which implies that $\mathfrak{q}\notin S_a$. Next, suppose that $q=2$, since $2$ is unramified in $\bK$, and the exact power of $2$ that divides $a$ is $4$, we find that $v_{\mathfrak{q}}(a)=4$. As a result, $v_{\mathfrak{q}}(4a)=6$ and thus, $\mathfrak{q}\notin S_a$. Finally, we get to the case when $q=3$. In this case, since $3\nmid a$, we find that $v_{\mathfrak{q}}(4a)=0$ and hence, $\mathfrak{q}\notin S_a$. We have thus shown that $S_a$ is empty and the result thus follows from Theorem \ref{type1bounds}. 
\end{proof}

\begin{proposition}\label{asymptotics for S(Y) and T(Y)}
    With respect to the notation above, the following assertions hold
    \begin{enumerate}
        \item\label{p1 of asymptotics for S(Y) and T(Y)} $\# \mathcal{T}(Y)\sim \frac{1}{32\zeta(2)} Y$, 
        \item\label{p2 of asymptotics for S(Y) and T(Y)} $\# \mathcal{S}(Y)\sim \frac{2}{\zeta(6)} Y$, 
        \item\label{p3 of asymptotics for S(Y) and T(Y)} $\op{lim}_{Y \to \infty} \frac{\# \mathcal{T}(Y)}{\# \mathcal{S}(Y)}=\frac{\zeta(6)}{64\zeta(2)}.$
    \end{enumerate}
\end{proposition}

\begin{proof}
We begin by proving \eqref{p1 of asymptotics for S(Y) and T(Y)}. At each prime $p$, we put a local condition $\Phi_p$ at $p$ defined as follows. For $p=2$, we say that $a$ satisfies $\Phi_2$ if $a\equiv 16\pmod{32}$. At $p=3$, say that $a$ satisfied $\Phi_3$ if $3\nmid a$. Finally for $p\geq 5$, $a$ satisfied $\Phi_p$ if $p^2\nmid a$. The local densities $\mathfrak{d}(\Phi_p)$ are as follows
 \[ \mathfrak{d}(\Phi_p)=\begin{cases}
    \frac{1}{32} & \text{ if }p=2;\\
    \frac{2}{3} & \text{ if }p=3;\\
    1-\frac{1}{p^2} & \text{ if } p\geq 5.
 \end{cases}\]
It follows from Proposition \ref{prop Phi asymptotic} that 
\[\begin{split}\lim_{Y \rightarrow \infty} \frac{\# \mathcal{T}(Y)}{Y} & =  \prod_p \mathfrak{d}(\Phi_p) \\
 & =\frac{1}{32}\times \frac{2}{3}\times \prod_{p\geq 5} \left(1-\frac{1}{p^2}\right) \\
 & = \frac{1}{48 (1-1/4)(1-1/9)\zeta(2)}=\frac{1}{32\zeta(2)}. \end{split}\]
\par The proof of part \eqref{p2 of asymptotics for S(Y) and T(Y)} is similar and is left to the reader. Part \eqref{p3 of asymptotics for S(Y) and T(Y)} follows from \eqref{p1 of asymptotics for S(Y) and T(Y)} and \eqref{p2 of asymptotics for S(Y) and T(Y)}. 
 \end{proof}

\begin{theorem}\label{main density result}
    There is a positive density of $6$-th power free integers $a$, such that 
    \[\dim_{\F_3}\op{Sel}^\varphi(E_a/\bK)\in \{r_3(L_a), r_3(L_a)+1\}.\] More precisely, 
    \[\liminf_{X\rightarrow \infty} \frac{\#\left\{a\in \mathcal{C}_1(X)\mid \dim_{\F_3}\op{Sel}^\varphi(E_a/\bK)\in \{r_3(L_a), r_3(L_a)+1\}\right\}}{\# \mathcal{C}_1(X)}\geq \frac{\zeta(6)}{64 \zeta(2)}.\]
\end{theorem}

\begin{proof}
    It is clear that 
    \[\begin{split}& \liminf_{X\rightarrow \infty} \frac{\#\left\{a\in \mathcal{C}_1(X)\mid \dim_{\F_3}\op{Sel}^\varphi(E_a/\bK)\in \{r_3(L_a), r_3(L_a)+1\}\right\}}{\# \mathcal{C}_1(X)} \\
    =& \liminf_{Y\rightarrow \infty} \frac{\#\left\{a\in \mathcal{S}(Y)\mid \dim_{\F_3}\op{Sel}^\varphi(E_a/\bK)\in \{r_3(L_a), r_3(L_a)+1\}\right\}}{\# \mathcal{S}(Y)}.\end{split}\]

    It follows from Proposition \ref{dim is 0 or 1} that 
    \[\mathcal{T}(Y)\subseteq \left\{a\in \mathcal{S}(Y)\mid \dim_{\F_3}\op{Sel}^\varphi(E_a/\bK)\in \{r_3(L_a), r_3(L_a)+1\}\right\}.\] Thus from Proposition \ref{asymptotics for S(Y) and T(Y)}, the result follows.
\end{proof}

\par For $a=16a'$ in $\mathcal{T}(Y)$, let $L_a:=\Q(\sqrt{-3},\sqrt{a})=\Q(\sqrt{-3},\sqrt{a'})$. Note that $L_a$ contains the quadratic subfields $\bK$, $\Q(\sqrt{a'})$ and $\Q(\sqrt{-3 a'})$. 

\begin{lemma}\label{lemma h_3 of L_a is 1}
    With respect to notation above, suppose that $r_3\left(\Q(\sqrt{-3a'})\right)=0$. Then, we have that $r_3(L_a)=0$.
\end{lemma}

\begin{proof}
       The Scholz' reflection principal \cite{sc} asserts that \[r_3(\Q(\sqrt{a'}))\in \{r_3\left(\Q(\sqrt{-3a'})\right), r_3\left(\Q(\sqrt{-3a'})\right)-1\}.\] Therefore, it follows from our assumptions that $r_3(\Q(\sqrt{a'}))=0$. Furthermore, a result by Hergoltz \cite{her} gives the formula \[r_3(L_{a'})=r_3\left(\Q(\sqrt{-3a'})\right)+r_3\left(\Q(\sqrt{a'})\right)=0.\] Note that $L_a=L_{a'}$, and hence, $r_3(L_a)=0$.
\end{proof}

\par Given $Y>0$, let \[\mathcal{T}'(Y):=\{a\in \mathcal{T}(Y)\mid \op{dim}_{\F_3} \op{Sel}^\varphi(E_a/\bK)\leq 1\}=\mathcal{T}(Y)\cap \mathcal{S}'(Y).\] 
\begin{lemma}
    For $a\in \mathcal{T}(Y)$, if $r_3(\Q(\sqrt{-3a'}))=0$, then, $a\in \mathcal{T}'(Y)$. 
\end{lemma}
\begin{proof}
    The result follows from Proposition \ref{dim is 0 or 1} and Lemma \ref{lemma h_3 of L_a is 1}.
\end{proof}

Let $Y>0$ be a real number and let $\mathcal{N}(Y)$ be the set of discriminants $\Delta_2 \in (-Y,0)$ such that 
\begin{itemize}
    \item $2 \nmid \Delta_2$,
    \item $3 \mid \Delta_2$ and
    \item $r_3\left(\Q(\sqrt{\Delta_2})\right)=0$.
\end{itemize}

\begin{lemma}\label{injection of N into T'}
    With respect to notation above, there is an injection
    \[\iota: \mathcal{N}\left(\frac{3Y}{16}\right) \hookrightarrow \mathcal{T}'(Y).\]
\end{lemma}
\begin{proof}
Let $\Delta_2 \in \mathcal{N}(Y)$, we write $\Q(\sqrt{\Delta_2})=\Q(\sqrt{-3a'})$, where $a'>0$ is square-free and not divisible by $2$ and $3$. Note that since $\Delta_2$ is odd, $a' \equiv 3 \pmod 4$. Setting $a:=16a'$, we observe that $a \in \mathcal{T}'(\frac{16Y}{3})$. Note that $\Delta_2$ uniquely determines $a$. Replacing $Y$ with $\frac{3Y}{16}$, we thus get an injection \[\iota: \mathcal{N}\left(\frac{3Y}{16}\right) \hookrightarrow \mathcal{T}'(Y).\]    
\end{proof}

\begin{theorem}\label{thm 4.7}
   For every $\epsilon>0$, there is a constant $C_\epsilon>0$ such that for all $Y>C_\epsilon$, 
\[\# \mathcal{N}(Y)\geq \left(\frac{1}{4\pi^2}-\epsilon\right) Y.\]
\end{theorem}
\begin{proof}
    Recall that $\mathcal{M}^-(Y)$ be the set of all  discriminants $\Delta_2 \in (-Y,0)$ such that $2 \nmid \Delta_2$ and $3 \mid \Delta_2$. Setting $\mathcal{N}'(Y):=\mathcal{M}^-(Y)\setminus \mathcal{N}(Y)$, observe that
\[\begin{split}\sum_{\Delta_2\in \mathcal{M}^-(Y)} h_3^*(\Delta_2)& = \sum_{\Delta_2\in \mathcal{N}(Y)} h_3^*(\Delta_2) +\sum_{\Delta_2\in \mathcal{N}'(Y)} h_3^*(\Delta_2) \\
& \geq \sum_{\Delta_2\in \mathcal{N}(Y)} 1+ \sum_{\Delta_2\in \mathcal{N}'(Y)} 3 \\ 
&= \# \mathcal{N}(Y)+ 3\# \mathcal{N}'(Y) \\
& = 3 (\# \mathcal{N}(Y)+ \# \mathcal{N}'(Y))- 2\# \mathcal{N}(Y)\\
& = 3 \# \mathcal{M}^-(Y)- 2\# \mathcal{N}(Y).
\end{split} \]

Therefore, we have shown that 
\[\# \mathcal{N}(Y)\geq \frac{1}{2}\left(3 \# \mathcal{M}^-(Y)-\sum_{\Delta_2\in \mathcal{M}^-(Y)} h_3^*(\Delta_2)\right).\]
Note that by Theorem \ref{DH modified thm}, \[\sum_{\Delta_2\in \mathcal{M}^-(Y)} h_3^*(\Delta_2)\sim 2 \# \mathcal{M}^-(Y),\] and that by Lemma \ref{M pm asymptotic 1/2 pi^2} 
\[\# \mathcal{M}^-(Y)\sim \frac{1}{2\pi^2} Y.\]Therefore, for every $\epsilon>0$, there is a constant $C_\epsilon>0$ such that for all $Y>C_\epsilon$, 
\[\# \mathcal{N}(Y)\geq \left(\frac{1}{4\pi^2}-\epsilon\right) Y.\]
\end{proof}

\begin{corollary}\label{last cor}
    The following assertions hold. 
    
    \begin{enumerate}
        \item\label{p1 of lastcor} For every $\epsilon>0$, there is a constant $C_\epsilon>0$ such that for all $Y>C_\epsilon$, 
\[\# \mathcal{T}'(Y)\geq \left(\frac{3}{64 \pi^2}-\epsilon\right) Y.\]
\item\label{p2 of lastcor} We have that 
\[\liminf_{Y\rightarrow \infty} \frac{\# \mathcal{T}'(Y) }{\# \mathcal{S}(Y)}\geq \frac{3 \zeta(6)}{128\pi^2} .\]
    \end{enumerate}
\end{corollary}
\begin{proof}
    Part \eqref{p1 of lastcor} is an immediate consequence of the Theorem \ref{thm 4.7} and Lemma \ref{injection of N into T'}. The Proposition \ref{asymptotics for S(Y) and T(Y)} asserts that 
    \begin{equation}\label{boring eqn 1}\# \mathcal{S}(Y)\sim \frac{2}{\zeta(6)} Y.\end{equation} Part \eqref{p2 of lastcor} is thus a consequence of part \eqref{p1 of lastcor} and \eqref{boring eqn 1}.
\end{proof}

\begin{theorem}\label{main Selmer density thm}
    There is a positive density of sixth power free integers $a$ such that $\dim_{\F_3}\op{Sel}^\varphi(E_a/\bK)\leq 1$. In greater detail, 
    \[\liminf_{X\rightarrow \infty} \frac{\# \left\{a\in \mathcal{C}_1(X)\mid \dim_{\F_3}\op{Sel}^\varphi(E_a/\bK)\leq 1 \right\}}{ \# \mathcal{C}_1(X)}\geq \frac{3 \zeta(6)}{128 \pi^2}.\]
\end{theorem}
\begin{proof}
    The result is an immediate consequence of Corollary \ref{last cor}.
\end{proof}

Note that for $a \in \mathcal{T}'(Y)$, one has that $a \notin \bK^{*2}$ and that $S_a = \emptyset$. Thus, for $a \in \mathcal{T}'(Y)$, we get that $\dim_{\F_3}\op{Sel}^3(E_a/\bK) \le 2$ by \cite[Corollary 3.23]{jhamajumdarshingavekar}.

\begin{theorem}\label{main 3 Selmer thm}
    There is a positive density of sixth power free integers $a$ such that $\dim_{\F_3}\op{Sel}_3(E_a/\bK)\leq 2$. More precisely, one has 
    \[\liminf_{X\rightarrow \infty} \frac{\# \left\{a\in \mathcal{C}_1(X)\mid \dim_{\F_3}\op{Sel}_3(E_a/\bK)\leq 2 \right\}}{ \# \mathcal{C}_1(X)}\geq \frac{3 \zeta(6)}{128 \pi^2}.\]
\end{theorem}
\begin{proof}
    The result follows immediately from Theorem \ref{main Selmer density thm} and \cite[Corollary 3.23]{jhamajumdarshingavekar}.
\end{proof}
\bibliographystyle{alpha}
\bibliography{references}

\begin{thebibliography}{Dav51b}

\bibitem[ABS22]{alpoge2022integers}
Levent Alp{\"o}ge, Manjul Bhargava, and Ari Shnidman.
\newblock Integers expressible as the sum of two rational cubes.
\newblock {\em arXiv preprint arXiv:2210.10730}, 2022.

\bibitem[BES20]{bes}
Manjul Bhargava, Noam Elkies, and Ari Shnidman.
\newblock The average size of the 3-isogeny selmer groups of elliptic curves y2=x3+k.
\newblock {\em Journal of the London Mathematical Society}, 101(1):299--327, 2020.

\bibitem[BS13]{BS3}
Manjul Bhargava and Arul Shankar.
\newblock The average size of the 5-selmer group of elliptic curves is 6, and the average rank is less than 1.
\newblock {\em arXiv preprint arXiv:1312.7859}, 2013.

\bibitem[BS15a]{BS1}
Manjul Bhargava and Arul Shankar.
\newblock Binary quartic forms having bounded invariants, and the boundedness of the average rank of elliptic curves.
\newblock {\em Ann. of Math. (2)}, 181(1):191--242, 2015.

\bibitem[BS15b]{BS2}
Manjul Bhargava and Arul Shankar.
\newblock Ternary cubic forms having bounded invariants, and the existence of a positive proportion of elliptic curves having rank 0.
\newblock {\em Ann. of Math. (2)}, 181(2):587--621, 2015.

\bibitem[Dav51a]{davenport1}
H.~Davenport.
\newblock On the class-number of binary cubic forms. {I}.
\newblock {\em J. London Math. Soc.}, 26:183--192, 1951.

\bibitem[Dav51b]{davenport2}
H.~Davenport.
\newblock On the class-number of binary cubic forms. {II}.
\newblock {\em J. London Math. Soc.}, 26:192--198, 1951.

\bibitem[DH69]{davenportheilbronn1}
H.~Davenport and H.~Heilbronn.
\newblock On the density of discriminants of cubic fields.
\newblock {\em Bull. London Math. Soc.}, 1:345--348, 1969.

\bibitem[DH71]{davenportheilbronn2}
H.~Davenport and H.~Heilbronn.
\newblock On the density of discriminants of cubic fields. {II}.
\newblock {\em Proc. Roy. Soc. London Ser. A}, 322(1551):405--420, 1971.

\bibitem[Her22]{her}
G.~Hergoltz.
\newblock {\" U}ber einen dirichletschen satz.
\newblock {\em Math. Zeitschrift}, 12:255—261, 1922.

\bibitem[JMS22]{jhamajumdarshingavekar}
Somnath Jha, Dipramit Majumdar, and Pratiksha Shingavekar.
\newblock $3 $-selmer group, ideal class groups and cube sum problem.
\newblock {\em arXiv preprint arXiv:2207.12487}, 2022.

\bibitem[Sch32]{sc}
A.~Scholz.
\newblock {\"U}ber die beziehung der klassenzahlen quadratischer k{\"o}rper zueinander.
\newblock {\em J. Reine Angw. Math.}, 166:201--203, 1932.

\bibitem[Top93]{top}
Jaap Top.
\newblock Descent by {$3$}-isogeny and {$3$}-rank of quadratic fields.
\newblock In {\em Advances in number theory ({K}ingston, {ON}, 1991)}, Oxford Sci. Publ., pages 303--317. Oxford Univ. Press, New York, 1993.

\bibitem[V\'71]{velu}
Jacques V\'{e}lu.
\newblock Isog\'{e}nies entre courbes elliptiques.
\newblock {\em C. R. Acad. Sci. Paris S\'{e}r. A-B}, 273:A238--A241, 1971.

\end{thebibliography}
\end{document}